\documentclass[11pt]{amsart}
\usepackage[utf8]{inputenc}
\usepackage[text={6.1in,8.5in},centering]{geometry}
\usepackage{amssymb,amsmath,amsthm}
\usepackage{tikz}

\usetikzlibrary{patterns}
\usetikzlibrary{fadings}

\newtheorem{thm}[equation]{Theorem}
\newtheorem*{thm*}{Theorem}
\newtheorem*{lem*}{Lemma}

\newtheorem{lem}[equation]{Lemma}
\newtheorem{prop}[equation]{Proposition}
\newtheorem*{prop*}{Proposition}

\theoremstyle{definition}

\numberwithin{equation}{section}

\DeclareMathOperator{\vol}{vol}
\DeclareMathOperator{\shell}{shell}

\newcommand{\ph}{\varphi}

\title{High-dimensional holeyominoes}
\author[G.~Malen]{Greg Malen}
\address[G.~Malen]{Department of Mathematics, Union College, Schenectady, NY, United States}
\email{maleng@union.edu}
\author[F.~Manin]{Fedor Manin}
\address[F.~Manin]{Department of Mathematics, UCSB, Santa Barbara, CA, United States}
\email{manin@math.ucsb.edu}
\author[E.~Rold\'an]{\'Erika Rold\'an}
\address[E.~Rold\'an]{Zentrum Mathematik, TU M\"unchen, Garching b. M\"unchen, Germany}
\email{erika.roldan@ma.tum.de}

\begin{document}
\begin{abstract}
     What is the maximum number of \textit{holes} enclosed by a $d$-dimensional polyomino built of $n$ tiles? Represent this number by $f_d(n)$.  Recent results show that $f_2(n)/n$ converges to $1/2$.  We prove that for all $d \geq 2$ we have $f_d(n)/n \to (d-1)/d$ as $n$ goes to infinity. We also construct polyominoes in $d$-dimensional tori with the maximal possible number of holes per tile.  In our proofs, we use metaphors from error-correcting codes and dynamical systems.
\end{abstract}

\maketitle

\section{Introduction}
In 1954 \cite{golomb1954checker}, Solomon W.\ Golomb defined a polyomino as a \emph{finite rook-connected subset of squares of the infinite checkerboard}.  Since then, polyominoes (often under the name \emph{lattice animals}) have become an important object of study in statistical physics, where they are used as a simple model of polymers and other accretion phenomena \cite{eden1958probabilistic, stauffer1978monte}.  The related problem of counting all polyominoes of a given size with certain geometric properties, which may help understand the probabilities of different behaviors of a system, leads to difficult open questions in combinatorics; see \cite{guttmann2009polygons} for an overview.  But for example, the asymptotic growth rate of the total number of polyominoes of a given size is well-understood, including in the higher-dimensional cases studied in this paper \cite{BBG, barequet2021improved}.

For our purposes, we view polyomino-hood as a constraint on subsets of the square tessellation of $\mathbb{R}^2$, or more generally the tessellation of $\mathbb{R}^d$ by unit cubes. (Here one imagines a rook in $d$-dimensional chess which can move in any coordinate direction.)  Subject to this constraint, we look to maximize the ``frothiness'' of the subset: the number of \emph{holes} (bounded connected components of the complement) that are enclosed inside. On the way, we will encounter ideas from error-correcting codes, sphere-packing, and dynamical systems.

\subsection{Definitions and results}
We now begin the formalities.  A \emph{$d$-polyomino} is a finite union of tiles (unit cubes) of the regular cubical tessellation of $\mathbb{R}^d$ which has connected interior.  In what follows we use the term \emph{polyomino} to refer to any $d$-polyomino for any $d\geq 2$. We also use the term \emph{$n$-omino} to refer to a polyomino with $n$ tiles.

A \emph{hole} in a polyomino is a bounded connected component of its complement. Viewed through the lens of algebraic topology, the number of holes in a polyomino is its top Betti number: the rank of its $(d-1)$st homology group.

For a fixed $d$, we define the function $f_d:\mathbb{N} \to \mathbb{N}$ such that $f_d(n)$ is the maximum number of holes in a $d$-dimensional $n$-omino.  The function $f_2(n)$ was studied exhaustively in \cite{KaR} and \cite{MaR}.  Our main theorem gives an asymptotic bound for $f_d(n)$:
\begin{thm} \label{thm:main}
    For every fixed $d \geq 2$,
    \[f_d(n)=\frac{d-1}{d}n-\Theta\bigl(n^{\frac{d-1}{d}}\bigr).\]
\end{thm}
For context, notice that in a $d$-dimensional checkerboard---let's say a tournament checkerboard with dark green and beige cubes---the union of the green cubes has about as many holes as tiles. However, this set is not a polyomino. Our theorem establishes the price of making sure the set is rook-connected: one needs to additionally fill in one in $2d-1$ beige cubes, and some extra near the boundary.  

We can avoid the issue of counting these extra boundary cubes, which generate the second order term in Theorem \ref{thm:main}, by reducing modulo a lattice in $\mathbb{Z}^d$.  The resulting torus is still tessellated by unit cubes, and we can define a \emph{toric polyomino} to be a rook-connected union of such cubes.  As one might hope, many such tori contain ``ideal polyominos'' with the maximal possible number of holes per tile:
\begin{thm} \label{thm:tori}
For every $d$, there is an infinite number of lattices $\Lambda \subset \mathbb{Z}^d$ such that there is a toric polyomino in $\mathbb{R}^d/\Lambda$ with $\frac{d}{2d-1}\det(\Lambda)$ tiles and $\frac{d-1}{2d-1}\det(\Lambda)$ holes.  Here $\det(\Lambda)$ is the covolume of $\Lambda$, i.e.~the volume of $\mathbb{R}^d/\Lambda$.
\end{thm}
In particular, this means that there are such tori with arbitrarily large volume.

The rest of the paper is structured as follows: Section \ref{sec:perfect} introduces ideas from error-correcting codes, which we use to construct polyominoes with a large number of holes in Section \ref{lower}.  In Section \ref{upper}, we complete the proof of Theorem \ref{thm:main} by proving an upper bound. In Section \ref{toripol}, we prove Theorem \ref{thm:tori} using ideas about linear dynamical systems.

\section{Perfect codes in the Lee metric}\label{sec:perfect}

An \emph{error-correcting code} is a scheme for transmitting information over noisy channels.  One introduces some amount of redundancy so that small errors can be corrected without ambiguity.  The study of error-correcting codes mainly focuses on finding efficient codes.  This is a kind of packing problem: an efficient code is one in which most pieces of data are close to one (and only one) code word.  Of course, the exact parameters depend on the types of errors one expects.

\begin{figure}[h]
\begin{center}
\begin{tikzpicture}[thick,scale=1.5]
    \coordinate (A1) at (0, 0);
    \coordinate (A2) at (0, 1);
    \coordinate (A3) at (1, 1);
    \coordinate (A4) at (1, 0);
    \coordinate (B1) at (0.3, 0.3);
    \coordinate (B2) at (0.3, 1.3);
    \coordinate (B3) at (1.3, 1.3);
    \coordinate (B4) at (1.3, 0.3);

\node[anchor=north] at (0,0) {$000= 0$};
\node[anchor=south] at (1.3, 1.3) {$111= 1$};


    \draw[fill=black!20,opacity=0.5] (A1) -- (A2) -- (A3) -- (A4);
    \draw[fill=black!20,opacity=0.6] (A1) -- (A2) -- (B2) -- (B1);
    \draw[fill=black!20,opacity=0.6] (B1) -- (B2) -- (B3) -- (B4);
    \draw[fill=black!20,opacity=0.6] (A3) -- (B3) -- (B4) -- (A4);
    
    \draw[ultra thick, ->] (A1) -- (A2);
    \draw[ultra thick, ->] (A1) -- (A4);
    \draw[ultra thick, ->] (A1) -- (B1);
    \draw[dashed, ultra thick, ->] (B3) -- (A3);
    \draw[dashed, ultra thick, ->] (B3) -- (B2);
    \draw[dashed, ultra thick, ->] (B3) -- (B4);

\end{tikzpicture}
\end{center}
\caption{Every three-bit string of 0's and 1's is either a code word  000 or 111, or is distance one from exactly one of the code words.}
\label{fig:cubeerr}
\end{figure}
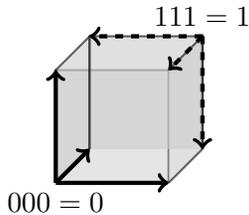
For example, say we want to transmit messages consisting only of zeros and ones.  The simplest way of introducing redundancy is to replace each 0 with the string 000 and each 1 with the string 111. We call these three-bit strings \emph{code words}. If the recipient sees the substring 101 instead of 000 or 111, they can guess that only one bit was flipped during transmission and the original bit was a 1. This is an example of a \emph{perfect code}: every possible string is either a code word or one error away from a unique code word.

The number of bit flips to get from one word to another defines a metric on $(\mathbb{Z}/2\mathbb{Z})^3$, called the \emph{Hamming metric}.  In this metric, the 1-balls around the code words are disjoint and cover all of $(\mathbb{Z}/2\mathbb{Z})^3$.  In other words, a perfect code is also a perfect packing by balls, or a tessellation.

This metaphor becomes clearer when we move to a slightly different setting: consider code words in the alphabet $\mathbb{Z}/q\mathbb{Z}$, for some $q>2$, and suppose that the errors that arise replace a letter $\alpha$ with $\alpha \pm 1$, rather than (as is perhaps more common) an arbitrary different letter.  The number of errors required to change from one point to another then gives us the \emph{Lee metric}, in other words the $\ell^1$ metric on $(\mathbb{Z}/q\mathbb{Z})^d$.

Perfect 1-error-correcting codes in the Lee metric were studied by Golomb and Welch \cite{GW}.  Such perfect codes can be ``unwrapped'' (lifted to $\mathbb{Z}^d$) to create perfect ball packings of $\mathbb{Z}^d$, or equivalently, tessellations of $\mathbb{R}^d$ by \emph{jacks}, polyominoes consisting of a tile and all its neighbors.

In the rest of this section we recall \cite[Theorem 3]{GW} and its proof:
\begin{thm} \label{thm:GW}
    Let $d \geq 1$ and $q=2d+1$.  Then there is a perfect 1-error-correcting code in the Lee metric on $(\mathbb{Z}/q\mathbb{Z})^d$.
\end{thm}
\begin{proof}
We take the set of code words $L_d$ to be the set of vectors $(a_1,\ldots,a_d) \in (\mathbb{Z}/q\mathbb{Z})^d$ which satisfy
\[\sum_{i=1}^d ia_i=0.\]
Then for every value of $a_2,\ldots,a_d$, there is exactly one value of $a_1$ such that $(a_1,\ldots,a_d) \in L_d$.  (Here we can replace $1$ with any $i$ which is relatively prime to $q$.)  In particular, this demonstrates a fact that will be useful later:
\begin{prop} \label{1/q}
Every $q \times 1 \times \cdots \times 1$ subset of $(\mathbb{Z}/q\mathbb{Z})^d$ contains exactly one element of $L$.
\end{prop}
In particular, the density of $L_d$ is $1/q$.

To show that $L_d$ is indeed a perfect 1-error correcting code, it is now enough to show that the jacks centered at points of $L_d$ cover all of $(\mathbb{Z}/q\mathbb{Z})^d$.  Let $(a_1,\ldots,a_d) \notin L_d$, and suppose that
\[\sum_{i=1}^d ia_i=k.\]
Then $L_d$ contains
\begin{align*}
    & (a_1,\ldots,a_k-1,\ldots,a_d) & \mbox{if }k \leq d \\
    & (a_1,\ldots,a_{q-k}+1,\ldots,a_d) & \mbox{if }k \geq d+1.
\end{align*}
Thus $(\mathbb{Z}/q\mathbb{Z})^d$ is covered by jacks centered at points of $L_d$.  Since a jack has volume $2d+1=q$, and the density of $L_d$ is $1/q$, this covering cannot have any overlaps.
\end{proof}

\section{Building polyominoes using perfect codes} \label{lower}

We will prove the lower bound of Theorem \ref{thm:main} by giving a near-optimal construction for all $d \geq 2$.  Later, we will prove the upper bound, showing that the construction is in fact optimal to first order.

The polyominoes we construct look like an almost-cubical shell surrounding an interior sponge which has as many holes as possible while still being connected.  Namely, given a domain $D \subset \mathbb{R}^d$, for example a large cube, we construct a polyomino $P(D)$ contained in $D$.  The cubes that touch the boundary of $D$, even if only in a corner---we call these the \emph{shell} of $D$---are all ``filled in'', i.e.~contained in $P(D)$.  The rest (the \emph{interior} of $D$) is filled according to a certain repeating pattern which is just dense enough that (1) every green checkerboard cube is filled in and (2) every filled-in cube is connected to the shell via a path of filled-in cubes.

We shift our tiling so that centers of cubes lie in $\mathbb{Z}^d$ and label cubes by the coordinates of their centers.  We refer to cubes for which the sum of the coordinates is 0 and 1 mod 2 as \emph{even} and \emph{odd} cubes, respectively.  We use the even cubes as the ``green'' checkerboard cubes of condition (1).

We satisfy condition (2) by filling in certain whole columns, so that every unit cube is either inside or adjacent to such a column.  To fill as few cubes as possible, we would like to choose as few columns as possible; this is achieved if every cube not inside one of the chosen columns is adjacent to exactly one of them.  The columns are indexed by elements of $\mathbb{Z}^{d-1}$, so we would like to find a subset $\widetilde{L_{d-1}} \subset \mathbb{Z}^{d-1}$ such that every point not in $\widetilde{L_{d-1}}$ neighbors exactly one element of $\widetilde{L_{d-1}}$.  Such a set is provided by Theorem \ref{thm:GW}: it is the lift to $\mathbb{Z}^{d-1}$ of the set $L_{d-1} \subset (\mathbb{Z}/(2d-1)\mathbb{Z})^{d-1}$ constructed there.

We can think of the inner part of $P(D)$ as the intersection of the interior of $D$ with a set of cubes indexed by a set $K_d \subset \mathbb{Z}^d$.  This consists of all even cubes and all cubes whose first $d-1$ coordinates are an element of $\widetilde{L_{d-1}}$.
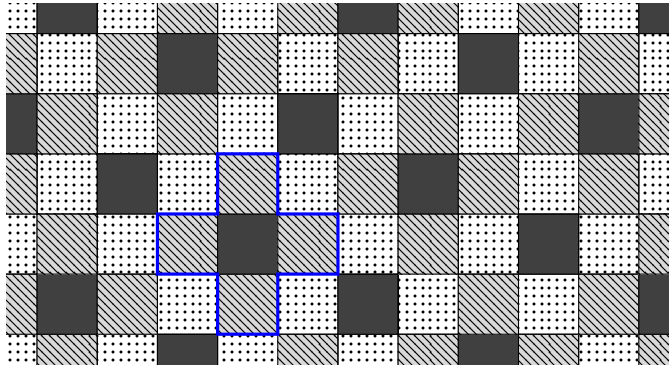
\begin{figure}
    \centering
    \begin{tikzpicture}[scale=0.8]
    \clip (-0.5,-0.5) -- (-0.5,5.5) -- (10.5,5.5) -- (10.5,-0.5) -- cycle;
    \newcounter{myint}
    \foreach \x in {-1,...,10}
        \foreach \y in {-1,...,5} {
            \pgfmathsetcounter{myint}{mod(\x+2*\y,5)}
            \ifnum\value{myint}=0
                \filldraw[fill=gray!50!black] (\x,\y) rectangle ++(1,1);
            \else
                \pgfmathsetcounter{myint}{\x+\y}
                \ifodd\value{myint}
                    \fill[color=gray!30!white] (\x,\y) rectangle ++(1,1);
                    \draw[pattern=north west lines] (\x,\y) rectangle ++(1,1);
                \else
                    \draw[pattern=dots] (\x,\y) rectangle ++(1,1);
                \fi
            \fi
        }
        \draw[blue, very thick] (3,0) -- (4,0) -- (4,1) -- (5,1) -- (5,2) -- (4,2) -- (4,3) -- (3,3) -- (3,2) -- (2,2) -- (2,1) -- (3,1) -- cycle;
    \end{tikzpicture}
    \caption{The set $K_3$, viewed from above.  Black columns are entirely in $K_3$ (these correspond to elements of $\widetilde{L_2}$); in shaded columns, odd-numbered layers are in $K_3$, and in dotted columns, even-numbered layers are in $K_3$.  One 2-dimensional jack is outlined.}
    \label{fig:L2}
\end{figure}
Then from Proposition \ref{1/q} we get:
\begin{prop} \label{q+1/2q}
Let $q=2d-1$.  Then every $q \times 1 \times \cdots \times 1 \times 2$ subset of $\mathbb{Z}^d$ contains exactly $q+1$ elements of $K_d$.
\end{prop}

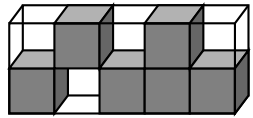
\begin{figure}[h]
\begin{center}
\begin{tikzpicture}[thick,scale=0.6]
    \foreach \x in {0,...,5}
        \foreach \y in {0,1,2} {
            \coordinate (A\x\y) at (\x, \y);
            \coordinate (B\x\y) at (\x+0.3,\y+0.4);
        }
    
    \draw (B00) -- (B02) -- (B52) -- (B50) -- cycle;
    \draw (A00) -- (B00) (A02) -- (B02) (A52) -- (B52) (A50) -- (B50);

    \draw[fill=black!50] (A00) -- (A01) -- (A11) -- (A10) -- cycle;
    \draw[fill=black!30] (A11) -- (A01) -- (B01) -- (B11) -- cycle;
    \draw[fill=black!60] (A11) -- (B11) -- (B10) -- (A10) -- cycle;

    \draw[fill=black!50] (A11) -- (A12) -- (A22) -- (A21) -- cycle;
    \draw[fill=black!30] (A22) -- (A12) -- (B12) -- (B22) -- cycle;
    \draw[fill=black!60] (A22) -- (B22) -- (B21) -- (A21) -- cycle;
    
    \draw[fill=black!50] (A20) -- (A21) -- (A31) -- (A30) -- cycle;
    \draw[fill=black!30] (A31) -- (A21) -- (B21) -- (B31) -- cycle;
    
    \draw[fill=black!50] (A30) -- (A31) -- (A41) -- (A40) -- cycle;
    \draw[fill=black!50] (A31) -- (A32) -- (A42) -- (A41) -- cycle;
    \draw[fill=black!30] (A42) -- (A32) -- (B32) -- (B42) -- cycle;
    \draw[fill=black!60] (A42) -- (B42) -- (B41) -- (A41) -- cycle;
    
    \draw[fill=black!50] (A40) -- (A41) -- (A51) -- (A50) -- cycle;
    \draw[fill=black!30] (A51) -- (A41) -- (B41) -- (B51) -- cycle;
    \draw[fill=black!60] (A51) -- (B51) -- (B50) -- (A50) -- cycle;
 
    \draw (A00) -- (A02) -- (A52) -- (A50) -- cycle;
\end{tikzpicture}
\end{center}
\caption{One possible filling pattern of a $5 \times 1 \times 2$ parallelepiped in $K_3$.}
\label{fig:parallel}
\end{figure}

In particular, the density of $K_d$ is
\[\frac{q+1}{2q}=\frac{2d}{4d-2}=\frac{d}{2d-1}.\]
Thus, by Proposition \ref{q+1/2q}, if the interior of $D$ (after subtracting the boundary cubes) is a union of $q \times 1 \times \cdots \times 1 \times 2$ parallelotopes, it contains $d-1$ holes for every $d$ filled-in tiles.  If in addition we choose $D$ to be close to isoperimetric, then
\[\vol(\shell(D))=O\bigl(\vol(D)^{\frac{d-1}{d}}\bigr).\]
The volume of $P(D)$ will then be
\[\vol(P(D))=\frac{d}{2d-1}\vol(D)+\frac{d-1}{2d-1}\vol(\shell(D))\]
and the number of holes in $P(D)$ will be
\[\frac{d-1}{2d-1}(\vol(D)-\vol(\shell(D)))=\frac{d-1}{d}(\vol(P(D))-\vol(\shell(D))).\]

A specific sequence of such polyominoes arises from the sequence $\{Q_i\}_{i=1,2,\ldots}$ of cubes of side length $2qi+2$.  In this case, every cube of $\shell(Q_i)$ has a $(d-1)$-face in the boundary, so that $\vol(\shell(Q_i)) \leq \vol_{s-1}(\partial Q_i)$, and
\[\vol_{d-1}(\partial Q_i) \leq 2d(\vol(Q_i))^{\frac{d-1}{d}} \leq 2d\left[\frac{2d-1}{d}\vol(P(Q_i))\right]^{\frac{d-1}{d}} \leq 4d(\vol(P(Q_i)))^{\frac{d-1}{d}}.\]
This shows that
\[f_d(\vol(P(Q_i))) \geq \frac{d-1}{d}\vol(P(Q_i))-4(d-1)\vol(P(Q_i))^{\frac{d-1}{d}}.\]
This proves the lower bound of Theorem \ref{thm:main} for a certain sequence of $n$, but not yet all $n$.

One way of completing the proof for all $n$ is to remark that
\[\vol(P(Q_{i+1}))-\vol(P(Q_i))=O\bigl(\vol(P(Q_i))^{\frac{d-1}{d}}\bigr),\]
and therefore we can obtain a good enough lower bound for all $n$ between $\vol(P(Q_i))$ and $\vol(P(Q_{i+1}))$ by adding cubes to $P(Q_i)$ arbitrarily.  However, we would like to obtain a more explicit bound.

To interpolate between the volumes of $P(Q_i)$ and $P(Q_{i+1})$, we use polyominoes $P(D)$ for sets $D$ of more complicated shape.  The interiors of these sets will still be composed of $q \times 1 \times \cdots \times 1 \times 2$ parallelotopes, which we will call \emph{fundamental parallelotopes}.  Let $m$ be any number between $(4d-2)^{d-1}i^d$ and $(4d-2)^{d-1}(i+1)^d$ (these are the numbers of fundamental parallelotopes in the interior of $Q_i$ and $Q_{i+1}$, respectively).  Let $D^0_m$ be the set composed of the interior of $Q_i$ together with some extra fundamental parallelotopes in the interior of $Q_{i+1}$ to make a total of $m$.  To choose these unambiguously, we add them in lexicographic order by the coordinates of the ``lowest'' corner.  In particular, we get from $D^0_m$ to $D^0_{m+1}$ by adding one fundamental parallelotope.  We define $D_m$ to be $D^0_m$ together with a one-cube-thick shell surrounding it.

We claim:
\begin{lem}
For every $(4d-2)^{d-1}i^d \leq m \leq (4d-2)^{d-1}(i+1)^d$,
\[\vol(\shell(D_m)) \leq \vol(\shell(Q_{i+1})).\]
\end{lem}
\begin{proof}
We will actually show that $\vol(\shell(D_m)) \leq \vol(\shell(D_{m+1}))$.  This suffices to prove the lemma since $D_{(4d-2)^{d-1}(i+1)^d}=Q_{i+1}$.

Notice that $D_m \subset D_{m+1}$.  Moreover, since fundamental parallelotopes are added in lexicographic order, if a cube labeled by $\mathbf x \in \mathbb Z^d$ is in $D_{m+1} \setminus D_m$, then every cube $\mathbf y \in D_m$ has $y_i<x_i$ for at least one $i$.

Let us proceed from $D_m$ to $D_{m+1}$, adding one cube at a time in lexicographic order.  At the time we add a cube $\mathbf x$, the positive orthant $\mathbf x+(\mathbb{Z}_{\geq 0})^d$ is completely empty.  We claim that the volume of the shell never decreases during this process: the shell gains the cube $\mathbf x$ and loses at most one cube.  Indeed, the only cube that might leave the shell is $(x_1-1,\ldots,x_d-1)$.  Any other cube which shares a face with $\mathbf x$ will remain in the shell since it also shares faces with other cubes in $\mathbf x+(\mathbb{Z}_{\geq 0})^d$.
\end{proof}

From this we have:
\begin{align*}
    \vol(\shell(D_m)) 
    &\leq 2d\bigl(\vol(D_m)^{1/d}+2d-1\bigr)^{d-1} \\
    &\leq 4d(\vol(P(D_m)))^{\frac{d-1}{d}}+O\bigl(\vol(P(D_m))^{\frac{d-2}{d}}\bigr).
\end{align*}
This shows that
\begin{equation} \label{eqn:lower}
    f_d(\vol(P(D_m))) \geq \frac{d-1}{d}\vol(P(D_m))-4(d-1)\vol(P(D_m))^{\frac{d-1}{d}}-O\bigl(\vol(P(D_m))^{\frac{d-2}{d}}\bigr).
\end{equation}
Now, the volumes of the $P(D_m)$ still don't cover every $n$.  However, the difference between $D_m$ and $D_{m+1}$ is at most one fundamental parallelotope together with a shell surrounding it, which means that it has a bound depending only on $d$:
\[\vol(P(D_{m+1}))-\vol(P(D_m)) \leq (2d+3) \cdot 3^{d-2} \cdot 4-(2d-2).\]
Thus interpolating by appending extra tiles arbitrarily to $P(D_m)$ extends the bound of \eqref{eqn:lower} to every $n$:
\[f_d(n) \geq \frac{d-1}{d}n-4(d-1)n^{\frac{d-1}{d}}-O\bigl(n^{\frac{d-2}{d}}\bigr).\]

\section{Upper bound} \label{upper}

We bound above the number of holes in a $d$-polyomino of volume $n$ by bounding the number of $(d-1)$-dimensional faces which border holes; this strategy to determine an upper bound for $f_d(n)$ is used in \cite{KaR} for the case $d=2$.  The resulting bound will match the first-order term of Theorem \ref{thm:main}.

For a $d$-polyomino $A$, the $(d-1)$-dimensional faces of the cubes that make it up are naturally partitioned into three sets: interior faces incident to two adjacent $d$-dimensional cubes; faces which bound holes; and exterior faces on the outer perimeter of $A$. We denote the sizes of these sets by $b(A)$, $p_{h}(A)$, and $p_{o}(A)$, respectively.

Suppose $A$ has $n$ $d$-cubes and encloses $h$ holes. Every $d$-cube has $2d$ $(d-1)$-dimensional faces.  Adding all of these together double-counts the interior faces, yielding that
\[2dn = p_{o}(A)+2b(A)+p_{h}(A).\]
The smallest holes are the size of a single $d$-cube, so $p_h(A) \geq 2dh$, and we can write
\begin{equation} \label{eqn:2dh}
    2dh \leq p_{h}(A) = 2dn - 2b(A) - p_{o}(A).
\end{equation}

We would like to find an upper bound on $h$ in terms of $n$ and $d$.  For this, we must find lower bounds for both $b(A)$ and $p_{o}(A)$.  Notice that $b(A)$ is (equivalently) the number of edges in the dual graph of $A$, so the minimal value for given $n$ is attained when the dual graph is a tree, and $b(A) \geq n-1$.  On the other hand, $p_o(A)$ is the perimeter of the union of $A$ and all its holes, which has volume at least $n+h$.   Since the most compact polyomino is a cube, this is bounded by
\[p_o(A) \geq 2d(n+h)^{\frac{d-1}{d}}.\]
Plugging these bounds into \eqref{eqn:2dh}, we get
\[h \le \frac{2dn-2(n-1)-2d(n+h)^{\frac{d-1}{d}}}{2d} \leq \frac{d-1}{d}n-(n+h)^{\frac{d-1}{d}}.\]
When $d=2$, this bound is essentially sharp, as discussed in \cite{MaR}.  On the other hand, when $d \geq 3$, the two quantities $b(A)$ and $p_o(A)$ cannot be simultaneously minimized, because the dual graph of the boundary of a large cube has many cycles.  It is likely that a tighter theoretical upper bound can be obtained by ``trading off'' the number of cycles and the size of the boundary.  Still, the naive bound gives the desired upper bound for Theorem \ref{thm:main}:
\[f_d(n) \leq \frac{d-1}{d}n-n^{\frac{d-1}{d}}.\]

\section{Toric polyominoes}\label{toripol}

We now study toric polyominoes and prove Theorem \ref{thm:tori}, which we restate here:
\begin{thm*}
For every $d \geq 2$, there is an infinite number of lattices $\Lambda \subset \mathbb{Z}^d$ such that there is a toric polyomino in $\mathbb{R}^d/\Lambda$ with $\frac{d}{2d-1}\det(\Lambda)$ tiles and $\frac{d-1}{2d-1}\det(\Lambda)$ holes.
\end{thm*}
We first see that this maximizes the number of holes in a toric polyomino of a given volume.  In the absence of an outer perimeter, equation \eqref{eqn:2dh} for a polyomino of volume $n$ becomes
\[2dh \leq 2dn-2b(A) \leq 2dn-2(n-1)=2(d-1)n+2,\]
and so $h \leq \frac{d-1}{d}n+\frac{1}{d}$.

To construct optimal examples, we take a lattice $\Lambda$ such that the set $K_d$ constructed in \S\ref{lower} is $\Lambda$-invariant, and let $P$ be the projection of $K_d$ to $\mathbb{R}^d/\Lambda$.  Then $P$ has the following features:
\begin{itemize}
\item The volume of $P$ is $\frac{d}{2d-1}\det(\Lambda)$.
\item Each cube not in $P$ is surrounded by cubes in $P$.
\end{itemize}
Therefore, $P$ is a toric polyomino with a maximal number of holes if and only if it is a toric polyomino, i.e.~if it is rook-connected.

Lattices $\Lambda$ with this property are easy to construct:
\begin{prop}
Let $\Lambda_0$ be the lattice consisting of points in $\widetilde{L_{d-1}} \times \mathbb{Z}$ for which the sum of the coordinates is even.  Then $K_d$ is $\Lambda_0$-invariant, and therefore invariant with respect to any sublattice of $\Lambda_0$.
\end{prop}
\begin{proof}
We refer to lattice points the sum of whose coordinates is even as \emph{even points}.  The set $K_d$ is the union of $\widetilde{L_{d-1}} \times \mathbb{Z}$ and the lattice of even points.  Therefore, any lattice whose action preserves these two lattices preserves $K_d$, in particular their intersection $\Lambda_0$.
\end{proof}
Thus the main goal of the section is to find lattices $\Lambda \subset \Lambda_0$ such that $K_d/\Lambda$ is rook-connected and investigate their properties.

To get a handle on this, recall that every rook-connected component of $K_d$ is a one-dimensional column (parallel to the $d$th coordinate direction) together with every other cube from each adjacent column.  Thus for $K_d/\Lambda$ to be rook-connected, we need to make sure that $\mathbb{R}^d/\Lambda$ is covered by the image of a single column and its neighbors.  We can think of this as a discrete analogue of either ergodicity or density for linear flows on the torus.

One of the most classical results in dynamical systems is Kronecker's density theorem: in $\mathbb{R}^2/\mathbb{Z}^2$, the trajectory of a flow along a line of irrational slope is dense.  More generally, in $\mathbb{R}^d/\mathbb{Z}^d$, the trajectory of a linear flow $(\alpha_1 t,\ldots,\alpha_d t)$ is dense as long as $\alpha_i/\alpha_j$ is irrational for every $i \neq j$.\pagebreak[2]

A discrete analogue of this follows immediately from the Chinese remainder theorem:
\begin{prop} \label{DiscErg}
Let $n_1,\ldots,n_{d-1}$ be pairwise relatively prime integers, and consider the linear flow $\ph(t)=(t/n_1,\ldots,t/n_{d-1},t)$ on $\mathbb{R}^d/\mathbb{Z}^d$.  The image of $\ph(t)$ contains every point of the form
\[(k_1/n_1,\ldots,k_{d-1}/n_{d-1},0), \qquad k_i \in \mathbb{Z}/n_i\mathbb{Z}.\]
\end{prop}

In these examples we fixed the torus and played with the direction of the flow.  In our main case of interest, the flow direction is fixed---it is the $d$th coordinate direction---and we can vary the torus.  Of course, we can switch between these points of view via a coordinate change, as in the following reformulation of Proposition \ref{DiscErg}.  From here on, we denote the standard basis vectors of $\mathbb{R}^d$ by $\mathbf e_1,\ldots,\mathbf e_d$.

\begin{figure}
    \centering
    \tikzfading[name=fade down,
        top color=transparent!0,
        bottom color=transparent!100]
    \tikzfading[name=fade up,
        top color=transparent!100,
        bottom color=transparent!0]
    
    \begin{tikzpicture}[z={(0.0,1.0)},x={(1.0,0.0)},y={(0.6,0.4)},
    dot/.style={circle, draw=blue, fill=blue, inner sep = 0pt, minimum size = 1mm}]
    \draw (4,2,3) -- (5,3,0) -- (0,3,0) -- (-1,2,3) -- cycle;
    \draw (0,0,0) -- (0,3,0);
    
    \foreach \x in {0,...,4}
        \foreach \y in {0,1,2} {
            \draw[blue,very thick] (\x,\y,0) -- (\x,\y,3);
        }
    \foreach \x in {0,...,5}
        \foreach \y in {0,...,3} {
            \draw[blue,very thick,path fading=fade down] (\x,\y,-0.2) -- (\x,\y,0);
            \draw[blue,very thick,path fading=fade up] (\x,\y,0) -- (\x,\y,0.3);
            \draw[blue,very thick,path fading=fade down] (\x-1,\y-1,2.7) -- (\x-1,\y-1,3);
            \draw[blue,very thick,path fading=fade up] (\x-1,\y-1,3) -- (\x-1,\y-1,3.2);
        }
    \draw[red, ultra thick] (2,1,0) -- (2,1,3);
    \draw[red, ultra thick, path fading=fade down] (2,1,0) -- (2,1,-0.2);
    \draw[red, ultra thick, path fading=fade up] (2,1,3) -- (2,1,3.2);
    \draw[red, ultra thick, path fading=fade down] (3,2,0) -- (3,2,-0.2);
    \draw[red, ultra thick, path fading=fade up] (1,0,3) -- (1,0,3.2);
    \draw[red, ultra thick, path fading=fade up] (3,2,0) -- (3,2,1.5);
    \draw[red, ultra thick, path fading=fade down] (1,0,3) -- (1,0,1.5);
    
    \draw (0,0,0) -- (-1,-1,3) -- (4,-1,3) -- (5,0,0) -- cycle;
    \draw (-1,-1,3) -- (-1,2,3) (5,0,0) -- (5,3,0) (4,-1,3) -- (4,2,3);
    \end{tikzpicture}
    \caption{A fundamental domain for the flow in Proposition \ref{DiscErg2} in 3D, with $n_1=5$, $n_2=3$, $c=3$.  A contiguous portion of the flow line is highlighted.}
    \label{fig:DiscErg}
\end{figure}
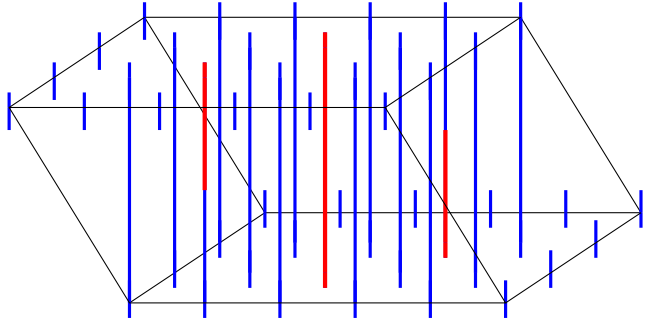
\begin{prop} \label{DiscErg2}
Let $n_1,\ldots,n_{d-1}$ be pairwise relatively prime integers, and consider the linear flow $\ph(t)=t\mathbf e_d$ on $\mathbb{R}^d/\Lambda_c$, where $\Lambda_c$ is the lattice generated by
\[n_1\mathbf e_1,\ldots,n_{d-1}\mathbf e_{d-1},c\mathbf e_d-\sum_{i=1}^{d-1} \mathbf e_i,\]
for any $c \in \mathbb{R}^+$.  The image of $\ph(t)$ is exactly
\[\{(k_1,\ldots,k_{d-1},t): k_i \in \mathbb{Z}/n_i\mathbb{Z}\text{ and }t \in \mathbb{R}/c\mathbb{Z}\}.\]
\end{prop}

This is already very close to what we want.  We get there using a further coordinate change:
\begin{prop} \label{DiscErgKd}
Let $d \geq 2$ and $q=2d-1$.  Let $n_1,\ldots,n_{d-1}$ be pairwise relatively prime integers such that $n_1$ is even, and $c$ an odd integer.  Fix the vectors
\begin{align*}
    \mathbf{u}_1 &= q\mathbf{e}_1 \\
    \mathbf{u}_i &= \mathbf{e}_i-i\mathbf{e}_1, & 2 &\leq i \leq d-1\text{, $i$ odd} \\
    \mathbf{u}_i &= \mathbf{e}_i+(q-i)\mathbf{e}_1, & 2 &\leq i \leq d-1\text{, $i$ even} \\
    \mathbf{u}_d &= \mathbf e_d,
\end{align*}
and let $\Lambda$ be the lattice in $\mathbb{R}^d$ generated by
\[n_1\mathbf u_1,\ldots,n_{d-1}\mathbf u_{d-1},c\mathbf e_d-\sum_{i=1}^{d-1} \mathbf u_i.\]\nopagebreak[0]
Then $\Lambda \subset \Lambda_0$ and $K_d/\Lambda$ is rook-connected.
\end{prop}
This proposition completes the proof of Theorem \ref{thm:tori}.
\begin{proof}
By applying the coordinate change $\mathbf e_i \mapsto \mathbf u_i$ to Proposition \ref{DiscErg2}, we see that the image of $\mathbb{R}\mathbf e_d$ in $K_d/\Lambda$ is
\[X=\{(x_1,\ldots,x_{d-1},t) : (x_1,\ldots,x_{d-1}) \in \Gamma\text{ and }t \in \mathbb{R}/c\mathbb{Z}\}\]
where $\Gamma$ is the lattice in $\mathbb{R}^{d-1} \times \{0\}$ generated by $\mathbf u_1,\ldots,\mathbf u_{d-1}$.

In fact, $\Gamma=\widetilde{L_{d-1}} \times \{0\}$.  To see this, notice that $\det(\Gamma)=q$ and that each $\mathbf u_i$ is an element of $\widetilde{L_{d-1}} \times \{0\}$.  Moreover, since $\mathbf u_1$ and $\mathbf u_d$ are odd, $n_1$ is even, $\mathbf u_2,\ldots,\mathbf u_{d-1}$ are even, and $c$ is odd, all the generators of $\Lambda$ are even.  Therefore $\Lambda \subset \Lambda_0$ and $K_d/\Lambda$ is well-defined.

From the fact that $\Gamma=\widetilde{L_{d-1}} \times \{0\}$, it follows that every cube in $K_d/\Lambda$ either is associated to a point of $X$ or is adjacent to a cube that is.  This proves that $K_d/\Lambda$ is rook-connected.
\end{proof}


To conclude, we make some remarks about the shapes of these tori.  We can get the infinite number of tori required by Theorem \ref{thm:tori} just by taking $n_1=2$, $n_2,\ldots,n_{d-1}=1$, and arbitrarily large $c$.  However, this solution is somehow unsatisfying, since the resulting tori are almost one-dimensional: they are long skinny tubes which the polyomino fills up by winding around them twice.

One way to measure the ``fatness'' of a Riemannian manifold is to estimate how big you can grow a metric ball around any point before it stops being topologically a ball.  For a flat torus $T=\mathbb{R}^d/\Lambda$, the supremal diameter of such balls is the \emph{systole} of $T$: the length of its shortest topologically nontrivial loop or, equivalently, the length of the shortest nonzero vector in the lattice $\Lambda$.  A large systole can be thought of as a sign that a $d$-torus is $d$-dimensional in an essential way.

The systole of a flat torus $T$ is bounded above by $C(d)(\vol T)^{1/d}$; to see this, scale the lattice to have covolume $1$ and notice that the size of the shortest nonzero vector is bounded above.  (The fact that a similar bound holds for arbitrary Riemannian metrics is much harder to prove and is an instance of Gromov's systolic inequality, see \cite{Guth}.)  The constant $C(d)^2$ is known as the \emph{Hermite constant} in dimension $d$.

The tori in our construction get within a multiplicative constant of this bound: if we choose $n_1,\ldots,n_{d-1}$ and $c$ to be large and close to each other, then the lattice $\Lambda$ is very nearly homothetic to $\Lambda_0$.  It would be interesting to know whether one can choose $\Lambda$ satisfying the conditions of Theorem \ref{thm:tori} whose systole is arbitrarily close to the universal upper bound.

\subsection*{Acknowledgements}
This project received funding from the European Union's Horizon 2020 research and innovation program under the Marie Sk\l odowska-Curie grant agreement No.~754462.  We would like to thank Matthew Kahle for his encouragement and comments on the paper.

\bibliographystyle{amsplain}
\bibliography{holeyominoes}
\end{document}